\documentclass[final,12pt]{elsarticle}
\usepackage{xcolor,soul,cancel}
\usepackage[color]{showkeys}
\definecolor{refkey}{rgb}{0,1,1}
\definecolor{labelkey}{rgb}{1,0,0}
\usepackage{amssymb}
\usepackage{amsmath}
\usepackage{amsthm}
\usepackage{graphicx}
\usepackage{float}
\journal{arXiv}

\newtheorem{thm}{Theorem}

\newtheorem{cor}{Corollary}
\newtheorem{prop}[thm]{Proposition}

\numberwithin{equation}{section}
\newcommand{\eq} [1] {\begin{equation}\label{#1}\quad}
\newcommand{\en} {\end{equation}}
\newcommand{\scal}[1]{\langle#1\rangle}
\newcommand{\norm}[1]{\left\Vert#1\right\Vert}
\newcommand{\abs}[1]{\left\vert#1\right\vert}

\newcommand{\C}{\mathbb C}

\newcommand{\R}{\mathbb R}

\newcommand{\diag}{\operatorname{diag}}

\newcommand{\im}{\operatorname{Im}}

\newcommand{\conv}{\operatorname{conv}}
\newcommand{\re}{\operatorname{Re}}
\newcommand{\tr}{\operatorname{trace}}
\newcommand{\cl}{\operatorname{cl}}
\newcommand{\Span}{\operatorname{Span}}
\newcommand{\St}{\operatorname{St}}
\begin{document}

\begin{frontmatter}

\title{On the Stampfli point of some operators and matrices   \tnoteref{support}}


\author{Thanin Quartz and Ilya M. Spitkovsky}
\address{Division of Science and Mathematics, New York  University Abu Dhabi (NYUAD), Saadiyat Island,
P.O. Box 129188 Abu Dhabi, United Arab Emirates}

\tnotetext[support]{The results are partially based on the Capstone project of the first named author [TQ] under the supervision of the second named author [IMS]. The latter was also supported in part by Faculty Research funding from the Division of Science and Mathematics, New York University Abu Dhabi.
\\ \hspace*{.5cm} Email addresses: thanin.quartz@nyu.edu [TQ]; ims2@nyu.edu, ilya@math.wm.edu,  imspitkovsky@gmail.com  [IMS]}

\begin{abstract}
The center of mass of an operator $A$ (denoted $\St(A)$, and called in this paper as the {\em Stampfli point} of $A$) was introduced by Stampfli in his Pacific J. Math (1970) paper as the unique $\lambda\in\C$ delivering the minimum value of $\norm{A-\lambda I}$. We derive some results concerning the location of $\St(A)$ for several classes of operators, including 2-by-2 block operator matrices with scalar diagonal blocks and 3-by-3 matrices with repeated eigenvalues. We also show that for almost normal $A$ its Stampfli point lies in the convex hull of the spectrum, which is not the case in general. Some relations between the property $\St(A)=0$ and Roberts orthogonality of $A$ to the identity operator are established.
\end{abstract}

\end{frontmatter}

\section{Introduction} 

Let $A$ be a bounded linear operator acting on a Hilbert space  $\mathcal H$. In case $\dim\mathcal H=n<\infty$ we will identify $\mathcal H$ with $\C^n$ and $A$ with its $n$-by-$n$ matrix representation in the standard basis $\{e_1,\ldots,e_n\}$  of $\C^n$. We will write $A\in B(\mathcal H)$ when the dimension of $\mathcal H$ is irrelevant and $A\in\C^{n\times n}$ to emphasize that it is finite. 

We will denote the norm, the spectrum, and the numerical range of $A$ as $\norm{A}$, $\sigma(A)$ and $W(A)$, respectively.  Recall that the latter (a.k.a. the {\em field of values}, or the {\em Hausdorff set} of $A$) is defined as  \eq{nr} W(A)=\{ \scal{Ax,x}\colon x\in\mathcal H,\ \norm{x}=1\}. \en
This notion goes back to classical papers by Toeplitz \cite{Toe18} and Hausdorff \cite{Hau}; \cite{GusRa} is a more recent standard reference for the properties of $W(A)$. It is known in particular that $W(A)$ is convex (Toeplitz-Hausdorff theorem), 
its closure $\cl{W(A)}$ contains $\sigma(A)$, and thus the convex hull of it:
\eq{nrs} \cl{W(A)}\supseteq \conv \sigma(A). \en
Operators for which the equality in \eqref{nrs} holds are called {\em  convexoid}; this class includes in particular all normal (and even hyponormal) operators. On the other hand, already for non-normal $A\in\C^{2\times 2}$ the inclusion in \eqref{nrs} is strict: $\conv\sigma(A)$ is then the line segment connecting the eigenvalues $\lambda_1,\lambda_2$ of $A$ while $W(A)$ is an elliptical disk with the foci at $\lambda_1,\lambda_2$ (the Elliptical Range Theorem). 

A variation of \eqref{nr} is the so called {\em maximal} numerical range $W_0(A)$ consisting of the limits of all convergent sequences $\scal{Ax_n,x_n}$ with unit vectors $x_n\in\mathcal H$ such that $ \norm{Ax_n}\to\norm{A}$.
This notion was introduced by J.~Stampfli in \cite{Sta70}, where it was also observed that $W_0(A)$ is a closed convex subset of $\cl{W(A)}$. For $A\in\C^{n\times n}$, $W_0(A)=W(B)$, where $B$ is the compression of $A$ onto the eigenspace of $A^*A$ corresponding to its maximal eigenvalue. 

In the same paper \cite{Sta70}, J.~Stampfli introduced the {\em center of mass} of $A$ as the (unique) value of $\lambda\in\C$ at which the minimum of $\norm{A-\lambda I}$ is attained. Since this term is overused, and to give credit where it is due, we will call this value of $\lambda$ the {\em Stampfli point} of $A$ and denote it $\St(A)$. 
In this notation, according to \cite[Corollary]{Sta70}:
\eq{stw0} \St(A)=\lambda  \text{ if and only if } 0\in W_0(A-\lambda I). \en  
Note that the statements in \eqref{stw0} are {\bf not} equivalent to $\lambda\in W_0(A)$, since the maximal numerical range does not behave nicely under shifts.

Several other observations made in \cite{Sta70} are as follows: 

If $A$ is normal (or even hyponormal), then $\St(A)$ is the center of the smallest circle circumscribing $\sigma(A)$. In general, $\St(A)$ lies in the closure of the numerical range of $A$ but not necessarily in the convex hull of its spectrum. It is also  mentioned in passing that the respective examples exist already when $A$ is nilpotent and $\dim\mathcal H=3$ but no specifics were provided. 

In this paper we further explore properties of the Stampfli point. Section~\ref{s:sdb} provides  an explicit formula for $\St(A)$ for $A$ unitarily similar to 2-by-2 block operator matrices with the diagonal blocks being scalar multiples of the identity operator $I$. This covers in particular quadratic operators, as well as tridiagonal matrices with constant main diagonal. In Section~\ref{s:ano} the property $\St(A)\in\conv\sigma(A)$ is extended from normal to so called almost normal operators. Sections~\ref{s:tri}--\ref{s:tps} are devoted to 3-by-3 matrices. An explicit procedure for computing $\St(A)$ when $\sigma(A)$ is a singleton $\{\lambda\}$ is outlined in Section~\ref{s:trim}, based on some auxiliary results established in Section~\ref{s:tri}. One of these results is the criterion for $\St(A)$ to coincide with $\lambda$. As a generalization of the latter, in Section~\ref{s:tps} we characterize matrices $A\in\C^{3\times 3}$ with a doubleton spectrum and $\St(A)$ coinciding with the multiple eigenvalue. Section~\ref{s:rorth} contains some observations on the relation between the Stampfli point of $A$ and the Roberts orthogonality of $A$ to $I$. Finally, several figures illustrating results of Sections~\ref{s:ano}--\ref{s:trim}   are presented in the Appendix.

\section{Operators with scalar diagonal blocks} \label{s:sdb}

Let us start with the simplest possible case, in which the answer is explicit and can be obtained directly by a straightforward computation. 

\begin{prop} \label{th:2by2} Let $\in\C^{2\times 2}$. Then $\St(A)=\tr A/2$. \end{prop} 
\begin{proof} Using a unitary similarity, put $A$ in an upper-triangular form:
\[ A =\begin{bmatrix} \lambda_1 & c \\ 0 & \lambda_2 \end{bmatrix}, \]
where $\{\lambda_1,\lambda_2\}=\sigma(A)$, $c\geq 0$. Then, for any $\lambda\in\C$:
\[ (A-\lambda I)^* (A-\lambda I)=\begin{bmatrix} \abs{\lambda_1-\lambda}^2 & c(\overline{\lambda_1}-\overline{\lambda}) 
\\ c(\lambda_1-\lambda) & c^2+\abs{\lambda_2-\lambda}^2 \end{bmatrix},\]
and so \begin{multline}\label{2max} 2\norm{A-\lambda I}^2= \abs{\lambda_1-\lambda}^2+\abs{\lambda_2-\lambda}^2+c^2\\ +
\sqrt{(\abs{\lambda_1-\lambda}^2-\abs{\lambda_2-\lambda}^2)^2+c^4+
	2c^2(\abs{\lambda_1-\lambda}^2+\abs{\lambda_2-\lambda}^2)}.\end{multline}
Elementary planar geometry shows that the minimum with respect to $\lambda$ in the right hand side of \eqref{2max} 
is attained when $\lambda$ is a midpoint of $[\lambda_1,\lambda_2]$.
\end{proof} 
Note that Proposition~\ref{th:2by2} can also be proved based on \eqref{stw0} and the explicit formula from \cite{HamS} for $W_0(A)$ in case of $2$-by-$2$ matrices. 


\begin{thm}\label{th:brs}Let $A\in \C^{n\times n}$ be unitarily similar to a matrix of the form 
\eq{brs} \begin{bmatrix} a_1 I_{n_1} & X \\ Y^* & a_2 I_{n_2} \end{bmatrix}, \en 
with $XY^*\in \C^{n_1\times n_1}$, $Y^*X\in \C^{n_2\times n_2}$ being normal. Then $\St(A)=(a_1+a_2)/2$. \end{thm}

Proposition~\ref{th:2by2} is of course a particular case of Theorem~\ref{th:brs} (corresponding to $n_1=n_2=1$), but at the same time also the main ingredient of its proof. 

\begin{proof} As was shown in \cite{BS04} (see the proof of Theorem~2.1 there), matrices under consideration are unitarily similar to direct sums of $\min\{n_1,n_2\}$ two-dimensional blocks $A_k$, all having $a_1,a_2$ as its diagonal entries, with $\abs{n_1-n_2}$ one-dimensional blocks, equal $a_1$ or $a_2$. According to Proposition~\ref{th:2by2}, $\St(A_k)=(a_1+a_2)/2$ does not depend on $k=1,\ldots,\min\{n_1,n_2\}$. 
Since for any $\lambda\in\C$, $\norm{A_k-\lambda I}\geq\abs{a_j-\lambda}$ ($j=1,2$), the value of $\St(A)$ for the whole matrix $A$ coincides with that of its blocks $A_k$. \end{proof} 

The normality of $XY^*, Y^*X$ holds in a trivial way if $Y=0$, i.e., $A$ is unitarily similar to 
\eq{mquad} \begin{bmatrix} \lambda_1 I & Z \\ 0 &  \lambda_2 I \end{bmatrix}, \en 
with $\{\lambda_1,\lambda_2\}=\sigma(A)$. This happens if and only if $A$ satisfies the equation \eq{quad} A^2+pA+qI=0 \en  with $p=-(\lambda_1+\lambda_2), \ q=\lambda_1\lambda_2$. 

Here is an infinite-dimensional analogue of this situation.

\begin{thm} Let $A\in B(\mathcal H)$ be a quadratic operator, i.e., \eqref{quad} holds for some $p,q\in\C$. Then $\St(A)=-p/2$.\label{th:quad} \end{thm} 

\begin{proof}As was observed in \cite{TsoWu}, for an operator $A$ satisfying \eqref{quad} there exists a partition $\mathcal H=\mathcal H_1\oplus\mathcal H_2$ with respect to which $A$ takes the form \eqref{mquad}. 
But then for any $\lambda\in\C$ 
\[ \norm{A-\lambda I}=\norm{A_0-\lambda I}, \text{ where } A_0=\begin{bmatrix} \lambda_1 & \norm{Z} \\ 0 &  \lambda_2 \end{bmatrix}\in \C^{2\times 2}, \] and so $\St(A)=\St(A_0)$. It remains to invoke Proposition~\ref{th:2by2}.
\end{proof} 

Note that in the setting of Theorem~\ref{th:quad} $\sigma(A)=\{\lambda_1,\lambda_2\}$, and according to \cite[Theorem 2.1]{TsoWu} $W(A)$ is an elliptical disk with the foci $\lambda_1,\lambda_2$ (possible degenerating into the line segment $[\lambda_1,\lambda_2]$). So, for quadratic $A$ the position of $\St(A)$ is defined by $\sigma(A)$ uniquely, and is indeed at the center of both $\sigma(A)$ and $W(A)$. This justifies to some extent the ``center of mass'' term for $\St(A)$ coined by Stampfli. 

Finally, if in \eqref{brs} $a_1=a_2$, then no conditions on $X,Y$ are needed and, moreover, the formula for $\St(A)$ holds in the inifinite dimensional setting. 

\begin{thm} \label{th:xy} Let $A\in B(\mathcal H)$. If there exists a subspace $\mathcal L$ of $\mathcal H$ such that compressions of $A$ onto $\mathcal L$ and its orthogonal complement $\mathcal L^\perp$ are both multiples of the identity by the same scalar $a$, then $\St(A)=a$. \end{thm} 
\begin{proof}It suffices to consider the case $a=0$. The operator $A$ then can be represented as $\begin{bmatrix} 0 & X \\ Y^* & 0 \end{bmatrix}$ with respect to the decomposition $\mathcal H =\mathcal L\oplus\mathcal L^\perp$. 

We need to show that the norm of $\begin{bmatrix} \lambda I & X \\ Y^* & \lambda I \end{bmatrix}$ attains its minimum with respect to $\lambda$ at $\lambda=0$. Equivalently, the rightmost point of the spectrum (or, which is the same in this case, the numerical range) of 
\[ \begin{bmatrix} \overline{\lambda} I & Y \\ X^* & \overline{\lambda I} \end{bmatrix}
\begin{bmatrix} \lambda I & X \\ Y^* & \lambda I \end{bmatrix}=
\abs{\lambda}^2I+\begin{bmatrix} YY^* & \overline{\lambda}X+\lambda  Y \\ \lambda X^*+\overline{\lambda}Y^* & X^*X \end{bmatrix} \] should be the smallest when $\lambda=0$. 

But this is indeed the case, simply because the norm of any block operator matrix $\begin{bmatrix} H_1 & Z \\ Z^* & H_2\end{bmatrix}$ with fixed positive semi-definite diagonal blocks is minimal when its off-diagonal blocks are equal to zero. \end{proof} 
\begin{cor}\label{co:tri} Let $A\in B(\mathcal H)$ be such that the entries $a_{ij}$ of its matrix in some orthonormal basis $\{f_j\}$ have the property: $a_{ij}=0$ if $i-j$ is even and different from zero; $a_{jj}:=a$ is independent of $j$. Then $\St(A)=a$. \end{cor}
Indeed, such $A$ meets conditions of Theorem~\ref{th:xy} with $\mathcal L =\Span\{f_j\colon j \text{ even}\}$. 

This corollary covers in particular tridiagonal matrices $A$ with constant main diagonal.  

\section{Almost normal operators} \label{s:ano}

We adopt the definition of almost normality which for $A\in\C^{n\times n}$ was introduced in \cite{Ikra11} as having at least $n-1$ pairwise orthogonal eigenvectors. We will therefore say that  $A\in B(\mathcal H)$ is {\em almost normal} if it has an invariant subspace $\mathcal L$ of codimension one such that $A|\mathcal L$ is normal. 

\begin{thm}\label{th:an} Let  $A\in B(\mathcal H)$ be an almost normal operator. Then \eq{stan} \St(A)\in\conv\sigma(A).\en \end{thm}

\begin{proof}According to the definition of almost normality, $A$ can be represented in the matrix form \eq{an} A=\begin{bmatrix} N & b \\ 0 & \mu\end{bmatrix} \en with respect to the partition $\mathcal H =\mathcal L\oplus\C$. Here $N\in B(\mathcal L)$ is a normal operator, $b$ can be identified with a vector in $\mathcal L$, and $\mu\in\C$. Observe that $\sigma(A)=\sigma(N)\cup\{\mu\}$. 
	
Suppose \eqref{stan} does not hold. By shifting, rotating and scaling $A$ we may then without loss of generality assume that $\St(A)=0$ while $\sigma(A)\subset\{z\colon \re z\geq 1\}$.
	
	Any unit vector $x\in\mathcal H$ can, up to inconsequential unimodular scalar mutliple, be written as $x=[\sqrt{t}\xi,\sqrt{1-t}]^T$. Here $\xi$ is a unit vector in $\mathcal L$ and $t\in[0,1]$. Then 
	\[ Ax=[\sqrt{t}N\xi+\sqrt{1-t}b,\sqrt{1-t}\mu]^T, \] and 
	\eq{nAx} \norm{Ax}^2=t\scal{N^*N\xi,\xi}+(1-t)(\norm{b}^2+\abs{\mu}^2)+2\sqrt{t(1-t)}\re\scal{N\xi,b}. \en
	Using the symbolic calculus for normal operators (see e.g. \cite[Chapter 12]{Ru91}),
	\[ N=\int_{\sigma(N)}\zeta\,dE(\zeta), \quad N^*N=\int_{\sigma(N)}\abs{\zeta}^2\,dE(\zeta),\] 
	where $E$ is the spectral decomposition of $N$, so \eqref{nAx} can be rewritten as
	\begin{multline*} \norm{Ax}^2=t\int_{\sigma(N)}\abs{\zeta}^2\,d\scal{E(\zeta)\xi,\xi}+(1-t)(\norm{b}^2+\abs{\mu}^2)\\
	+2\sqrt{t(1-t)}\re\scal{\int_{\sigma(N)}\zeta\,dE(\zeta)\xi,b}.\end{multline*} 
	Considering vectors of the form $U\xi$ along with $\xi$, where \[ U=\int_{\sigma(N)}\phi(\zeta)\,dE(\zeta), \quad \abs{\phi}=1 \text{ on } \sigma(N), \]
	are unitary operators commuting with $N$, we conclude that $\norm{A}^2$ is the supremum of 
	\eq{noa} t\int\displaylimits_{\sigma(N)}\abs{\zeta}^2d\scal{E(\zeta)\xi,\xi}+(1-t)(\norm{b}^2+\abs{\mu}^2)
	+2\sqrt{t(1-t)}\int\displaylimits_{\sigma(N)}\abs{\zeta}\abs{\scal{dE(\zeta)\xi,b}}\en 	
	taken over $t\in [0,1]$ and unit vectors $\xi\in\mathcal L$. 
	
	A similar reasoning applied to $A-I$ in place of $A$ yields the conclusion that $\norm{A-I}^2$ is the supremum of  
	\begin{multline} \label{noa1} t\int_{\sigma(N)}\abs{\zeta-1}^2d\scal{E(\zeta)\xi,\xi}+(1-t)(\norm{b}^2+\abs{\mu-1}^2)\\
	+2\sqrt{t(1-t)}\int_{\sigma(N)}\abs{\zeta-1}\abs{\scal{dE(\zeta)\xi,b}}\end{multline}
	over the same set. 
	
	Since $\re\zeta,\re\mu\geq 1$, we have \[ \abs{\zeta}^2-\abs{\zeta-1}^2\geq 1, \quad 
	\abs{\mu}^2-\abs{\mu-1}^2\geq 1, \quad \abs{\zeta}\geq\abs{\zeta-1}, \] and so the difference between \eqref{noa} and \eqref{noa1} is not smaller than 
	\[ t\int_{\sigma(N)}\, d\scal{E(\zeta)\xi,\xi}+1-t =1. \]
	But then $\norm{A}>\norm{A-I}$, in contradiction with $\St(A)=0$. \end{proof} 

For convenience of readers interested in finite-dimensional setting only, let us provide a (shorter and more elementary) adaptation of this proof to the case of almost normal $A\in\C^{n\times n}$. 
\smallskip

{\em Proof of Theorem~\ref{th:an} in the finite dimensional setting.}	

Via an appropriate unitary similarity, for an almost normal $A\in\C^{n\times n}$ its representation \eqref{an} can be further specified to 
\eq{anf} A=\begin{bmatrix} \lambda_1 & 0 & \ldots & 0 & b_1 \\
	0 & \lambda_2 & \ldots & 0 & b_2 \\ \vdots & & \ddots  & &  \vdots \\ 0 & \ldots & 0 & \lambda_{n-1} & b_{n-1}\\
	0 & \ldots & \ldots & 0 & \mu \end{bmatrix}, \en 
with $b_j\geq 0$, $j=1,\ldots,n-1$. 

For $\xi=[\xi_1,\ldots,\xi_n]^T\in\C^n$,
denote $\abs{\xi_j}=t_j$, and without loss of generality set $\xi_n=t_n$. Since the $j$-th entry of $A\xi$ is
$\lambda_j\xi_j+b_jt_n$, its maximal absolute value is attained when \eq{argf} \arg\xi_j=-\arg\lambda_j \text{ if } \lambda_j b_jt_n\neq 0, \quad j=1,\ldots,n-1.\en  	
So, condition \eqref{argf} holds for $\xi$ maximazing $\norm{A\xi}/\norm{\xi}$. Therefore, for such $\xi$
\eq{scalalm}
\scal{A\xi,\xi}=\sum_{j=1}^{n-1}\left(\lambda_jt_j^2+\overline{\xi_j}b_jt_n\right)+\mu t_n^2
 \en 
is a linear combination of the points in $\sigma(A)$ with non-negative  coefficients not all equal to zero. 

Suppose now that \eqref{stan} does not hold. Shifting and rotating $A$ (as it was done in the proof for the general setting) we may without loss of generality assume that $\St(A)=0$ while $\lambda_1,\ldots, \lambda_{n-1},\mu$ all have positive real parts. According to \eqref{scalalm}, $\scal{A\xi,\xi}$ then also has a positive real part and is therefore different from zero. This is in contradiction with \eqref{stw0}. \qed

Recall that an almost normal matrix is {\em pure} if it is unitarily irreducible. According to \cite[Theorem 2.1]{MoSp} this is the case if and only if in \eqref{anf} all $\lambda_j$ are distinct and $b_j\neq 0$,
$j=1,\ldots n-1$.
The following refinement of the inclusion \eqref{stan} in the finite dimensional case is of some interest.
\begin{cor}\label{co:stanf} Let $A$ be as in \eqref{anf} with all $b_j$ different from zero. Then $\St(A)$ lies in the relative interior of $\conv\{\lambda_1,\ldots,\lambda_{n-1},\mu\}$.\end{cor} 
\begin{proof} We will continue using the notation $\abs{\xi_j}=t_j$ along with the convention $\xi_n=t_n$. Since $\norm{A}\geq\sqrt{\abs{\lambda_j}^2+b_j^2}$, under the condition imposed on $A$ we have $\norm{A}>\abs{\lambda_j}$  for all $j=1,\ldots,n-1$. So, $\xi$ cannot be orthogonal to $e_n$, i.e.,  $t_n>0$. 

Moreover, since 
\[  A^*A=\begin{bmatrix} \abs{\lambda_1}^2 & 0 & \ldots & 0 & \overline{\lambda_1}b_1 \\
0 & \abs{\lambda_2}^2 & \ldots & 0 & \overline{\lambda_2}b_2 \\ \vdots & & \ddots  & &  \vdots \\ 0 & \ldots & 0 & \abs{\lambda_{n-1}}^2 & \overline{\lambda_{n-1}}b_{n-1}\\
\lambda_1 b_1 & \lambda_2 b_2 & \ldots & \lambda_{n-1}b_{n-1} & \mu \end{bmatrix}, \]
the $j$-th entry of $A^*A\xi$ is $\abs{\lambda_j}^2\xi_j+\overline{\lambda_j}b_jt_n$ ($j=1,\ldots,n-1$). So, $t_j$ can only equal zero if $\lambda_j=0$.  

With these observations at hand, suppose that $\St(A)$ is not in the relative interior of $\conv\{\lambda_1,\ldots,\lambda_{n-1},\mu\}$. As in the proof above,  we may assume that $\St(A)=0$, but instead of $\re\lambda_j, \re \mu$ being positive we have a weaker condition $\re\lambda_j\geq 0, \re\mu\geq 0$, with at least one of the inequalities strict. Since now all $t_j$ and $b_j$ are non-zero, we still may conclude from \eqref{scalalm} that $\re\scal{A\xi,\xi}>0$, in contradiction with \eqref{stw0}.   
\end{proof} 
Note that all $A\in\C^{2\times 2}$ are either normal or pure almost normal, while  Proposition~\ref{th:2by2} shows that for such matrices the statement of Corollary~\ref{co:stanf} holds in both cases. However, already for $n=3$ there exist unitarily reducible and almost normal (or even normal) matrices $A$ with $\St(A)$ lying on the relative boundary of $\conv\sigma(A)$.   	

\section{3-by-3 matrices with singleton spectra. Auxiliary statements}\label{s:tri} 

Let $A\in\C^{3\times 3}$ be such that its spectrum is a singleton: $\sigma(A)=\{\lambda\}$. Then, up to a unitary similarity, \eq{3by3} A=\begin{bmatrix} \lambda & x & y\\ 0 & \lambda & z \\0 & 0 &\lambda\end{bmatrix}. \en 

\begin{prop}\label{th:3by30}The matrix \eqref{3by3} has $\St(A)=\lambda$ if and only if $xyz=0$. \end{prop}
Note that for matrices \eqref{3by3} $\conv\sigma(A)=\{\lambda\}$. So, condition $xyz=0$ is actually the criterion for $St(A)\in\conv\sigma(A)$ to hold in this case. 

\begin{proof} To make use of \eqref{stw0}, compute 

\[  (A-\lambda I)^* (A-\lambda I)= \begin{bmatrix} 0 & 0 & 0 \\ 0 & \abs{x}^2 & \overline{x}y \\
0 & x\overline{y} & \abs{y}^2+\abs{z}^2\end{bmatrix}. \]

If $xy\neq 0$, then the maximal eigenvalue of $ (A-\lambda I)^* (A-\lambda I)$ is simple, and the respective eigenvector is $\xi=[0,\xi_2,\xi_3]^T$ with $\xi_2,\xi_3\neq 0$. But then $\scal{(A-\lambda I)\xi,\xi}=z\overline{\xi_2}\xi_3$.
So, condition $W_0(A-\lambda I)\ni 0$ holds in this case if and only if $z=0$.

On the other hand, if $xy=0$, then one of the standard basis vectors $e_2$ or $e_3$ maximizes the norm of  
$A-\lambda I$ while $\scal{(A-\lambda I)e_j,e_j}=0$. 
\end{proof} 

Our goal in the next Section~\ref{s:trim} is to compute $\St(A)$ for matrices \eqref{3by3} with $xyz\neq 0$. An intermediate step in this direction is a characterization of such matrices with $\lambda=1$ and $\St(A)=0$. By an additional (in this case, diagonal) unitary similarity we can arrange that $x,z>0$. 

\begin{prop}\label{th:crit10} Let the matrix $A$ be given by \eqref{3by3} in which $\lambda=1$ and $x,z>0$. Let also $\St(A)=0$. Then $y<0$ and \eq{res} \det \begin{bmatrix}
		a_1 & 0 & b_1 & 0 & 0 \\ 
		a_2 & a_1 & b_2 & b_1 & 0 \\ 
		a_3 & a_2 & b_3 & b_2 & b_1 \\ 
		a_4 & a_3 & 0 & b_3 & b_2 \\ 
		0 & a_4 & 0 & 0 & b_3
	\end{bmatrix}=0, \en
where
 \eq{ab} \begin{aligned} & a_1 = 2x,\ a_2 = -3xz, \ a_3 = xz^2,\ a_4 = 2y - xz, \ 
	b_1 = 4x^2 + y^2 + z^2 - xyz, \\ & b_2 = -(z^3 + x^2z + y^2z + 6xy - xyz^2), \ b_3 = x^2 + 4y^2 + z^2 - xyz.\end{aligned} \en 
\end{prop} 
\begin{proof}According to \eqref{stw0}, there exists a unit vector $\xi=[\xi_1,\xi_2,\xi_3]^T$ such that $\norm{A\xi}=\norm{A}$ and $\scal{A\xi,\xi}=0$. Let us denote $\abs{\xi_j}=t_j$; without loss of generality $\xi_3\geq 0$ and therefore $\xi_3=t_3$. 

A direct computation shows that 
\begin{align*}
A\xi = \begin{bmatrix}
\xi_1 + x \xi_2 + y \xi_3 \\ \xi_2 + z\xi_3 \\ \xi_3
\end{bmatrix}
\end{align*}
and so 
\eq{norm}
\norm{A \xi}^2 = t_3^2 + \abs{\xi_1 + x \xi_2 + y \xi_3}^2 + \abs{\xi_2 + z \xi_3}^2.
\en
With $\xi_2,\xi_3$ and $\abs{\xi_1}$ being fixed, the right hand side of \eqref{norm} is maximal when 
\eq{arg}
\arg{\xi_1} = \arg{(x \xi_2 + y \xi_3)},
\en
and thus can be rewritten as 
\begin{align*}
t_1^2 + t_3^2 + \abs{x \xi_2 + y \xi_3}^2 + \abs{\xi_2 + z \xi_3}^2 + 2t_1 \abs{x \xi_2 + y \xi_3}.
\end{align*}
At the same time 
\begin{align*}
\scal{A\xi,\xi} = 1 + \bar{\xi_1}\left(x \xi_2 + y\xi_3 \right) + z\bar{\xi_2}\xi_3= 1 + \bar{\xi_1}\left(x \xi_2 + yt_3 \right) + z\bar{\xi_2}t_3, 
\end{align*}
which due to \eqref{arg} implies 
\eq{sc0}
\scal{A\xi,\xi} = 1 + t_1\abs{x\xi_2 + yt_3} + z\overline{\xi_2}t_3.  
\en
Since $z>0$ and $t_3\geq 0$, condition $\scal{A\xi,\xi}=0$ can hold only if $t_3>0$, $\xi_2<0$ (and thus $\xi_2=-t_2$).

Furthermore, $\xi=[\xi_1,-t_2,t_3]^T$ is an eigenvector of $A^*A$ corresponding to its eigenvalue $\sigma:=\norm{A}^2$. Since 
\eq{a*a} A^*A= \begin{bmatrix} 1 & x & y \\ x & x^2+1 & xy+z \\ \overline{y} & x\overline{y}+z & \abs{y}^2+z^2+1 \end{bmatrix}, \en
equating the second entries of $A^*A\xi$ and $\sigma\xi$ we have in particular 
\eq{y} x\xi_1-t_2(x^2+1)+t_3(xy+z)=-\sigma t_2. \en
Due to \eqref{arg}, $\xi_1=\mu(-xt_2+yt_3)$ with some $\mu\geq 0$. Plugging this into \eqref{y} and 
taking the imaginary parts:
\[ \im \left(xyt_3(\mu+1)\right)= xt_3(\mu+1)\im y=0. \] 
So, $y\in\R$. According to \eqref{arg}, then also $\xi_1\in\R$. 

Suppose for a moment that $y>0$. Then \eqref{a*a} shows that the matrix $A^*A$ is entry-wise positive. By the Perron theorem, its maximal eigenvalue is simple and the respective eigenvectors have entries with coinciding arguments. This is in contradiction with $\xi_2,\xi_3$ having opposite signs. This proves the first assertion of the proposition: $y<0$.

Returning to \eqref{arg} again, we see that $\xi_1<0$. So finally $\xi=[-t_1,-t_2,t_3]^T$ with all $t_j$ positive. 

Plugging this into \eqref{sc0} and recalling that $t_1^2+t_2^2+t_3^2=1$:
\eq{s1} t_1^2+t_2^2+t_3^2+xt_1t_2-yt_1t_3-zt_2t_3=0. \en 
The collinearity of $A^*A\xi$ and $\xi$ yields two more homogeneous second degree equations in $t_j$: 
\eq{s2} xt_1^2-xt_2^2+x^2t_1t_2-(xy+z)t_1t_3+yt_2t_3=0 \en 
and 
\eq{s3} yt_1^2-yt_3^2+(xy+z)t_1t_2-(y^2+z^2)t_1t_3+xt_2t_3=0.\en 
Dividing \eqref{s1}--\eqref{s3} by $t_3^2$ and introducing new variables $u=t_1/t_3$, $v=t_2/t_3$: 
\eq{sys} 
\begin{aligned}
u^2+v^2+xuv-yu-zv+1 & = 0, \\
xu^2-xv^2+x^2uv-(xy+z)u+yv & =0, \\
yu^2+(xy+z)uv-(y^2+z^2)u+xv-y & =0.
 \end{aligned} 
\en 
From the first two equations of \eqref{sys} it follows that 
\eq{u} u= \frac{-2xv^2+(xz+y)v-x}{z}, \en
while the first, multiplied by $y$ and subtracted from the third, becomes 
\eq{s4} yv^2-zuv+z^2u-(x+yz)v+2y=0. \en 
Plugging \eqref{u} into \eqref{s4} and the first equation of \eqref{sys} yields the following system of two equations in one variable $v$: 
\eq{sys1}
\begin{aligned}
2xv^3 - 3xzv^2 + xz^2v + 2y -xz &= 0, \\
4x^2v^4 - \left(4xy + 6x^2z \right)v^3 + \left(4x^2 + y^2 + z^2 + 5xyz + 2x^2z^2 \right)v^2 &  \\ -\left(3x^2z + y^2z + z^3 + xyz^2 + 2xy \right)v + (x^2 + z^2 + xyz) & = 0.  
\end{aligned} \en 
By some additional equivalent transformations aimed at lowering the degrees of polynomials involved, the system \eqref{sys1} can be reduced to 
\eq{sys2} \begin{aligned} 
2xv^3 - 3xzv^2 + xz^2v + 2y -xz & = 0, \\
\left(4x^2 + y^2 + z^2 - xyz \right)v^2 - \left(z^3 + x^2z + y^2z + 6xy - xyz^2 \right)v & \\ + 
\left(x^2 + 4y^2 + z^2 - xyz \right) & = 0. \end{aligned} \en 
It remains to observe that the determinant in \eqref{res} is the resultant of the left hand sides of \eqref{sys2}
and thus its equality to zero is equivalent to the system \eqref{sys2} being consistent.   
\end{proof} 

\section{3-by-3 matrices with singleton spectra. Main result} \label{s:trim}

Admittedly, Proposition~\ref{th:crit10} does not look constructive. Nevertheless, it serves as the main ingredient in the construction of a procedure allowing to compute $\St(A)$ for all matrices of the form \eqref{3by3}. Due to Proposition~\ref{th:3by30}, only the case $xyz\neq 0$ needs to be considered. 

\begin{thm}\label{th:comp}Let $A$ be given by \eqref{3by3} with $xyz\neq 0$. Then $\St(A)=\lambda+\zeta$, where \eq{arg1} \arg\zeta=\arg(x\overline{y}z) \en  and $\abs{\zeta}$ is a positive root of the 5-th degree polynomial

\begin{multline}\label{s}
P_A(s)= \ 4 s^5 (u^2+v^2)(u^6 + 3 u^4 (v^2 + w^2) + (v^2 + w^2)^3 + 3 u^2 (v^4 - 7 v^2 w^2 + w^4)) \\ + 4 s^4 uvw (4 u^6 + 6 u^4 (2 v^2 - 3 w^2) + (v^2 + w^2)^2 (4 v^2 + w^2) + 6 u^2 (2 v^4 - 6 v^2 w^2 + w^4))  \\
+ 3 s^3 u^2 w^2 (2 u^6 + 7 v^6 + 13 v^4 w^2 + 6 v^2 w^4 + u^4 (11 v^2 - 5 w^2) + 2 u^2 (8 v^4 - 14 v^2 w^2 + w^4)) \\ + s^2 u^3 v w^3 (9 u^4 + 7 v^4 + 18 v^2 w^2 + 6 w^4 + 4 u^2 (4 v^2 - 3 w^2))\\ + s u^4 v^2 w^4 (-5 v^2 + 3 w^2) - 3u^5 v^3 w^5.  
\end{multline} \end{thm}
Here for the notational convenience we have set \eq{uvw} u=\abs{x},\ v=\abs{y}, \text{ and } w=\abs{z}.\en 

\begin{proof}Denoting $\St(A)-\lambda :=\zeta$, from Proposition~\ref{th:3by30} we have $\zeta\neq 0$. So, we may consider the matrix 
\eq{StD} B=-\zeta^{-1}(A-(\lambda+\zeta)I)= \begin{bmatrix}1 & -x/\zeta & -y/\zeta \\ 0 & 1 & -z/\zeta \\ 0 & 0 & 1\end{bmatrix}, \en 
for which $\St(B)=0$. A diagonal unitary similarity can be applied to put $B$ in the form 
\eq{B1} \begin{bmatrix}1 & u/\abs{\zeta} & -ye^{-i(\arg x+\arg z-2\arg\zeta)}/\zeta \\ 0 & 1 & w/\abs{\zeta} \\ 0 & 0 & 1\end{bmatrix}, \en
thus making Proposition~\ref{th:crit10} applicable. Consequently, the right upper entry of the matrix \eqref{B1} is negative. This is equivalent to \eqref{arg1} and also allows to rewrite this entry as $-v/\abs{\zeta}$. 

Replacing $x,y,z$ in \eqref{ab} by $u/s, -v/s$ and $w/s$, respectively, and expanding the determinant in \eqref{res}, we arrive at \eqref{s}.
\end{proof} 
Note that the polynomial \eqref{s} has an odd number of positive roots, and so there is at least one. If such a root is unique, it of course delivers the correct value of $\abs{\zeta}$. In case of several such roots, the inclusion $W_0(B)\ni 0$ should be checked in order to choose the ``right'' one. In all our numerical experiments this was the smallest positive root, but we do not have a proof that this is always the case. Numerical examples are shown below, one with a unique root for \eqref{s} and one where there are several positive roots. Direct computations of $\St(X)$ for various matrices $X$ throughout the paper were carried out by the Mathematica script \textit{NMinimize}. 

\smallskip
{\bf Example 1.}
Consider A as follows, 
\begin{align*}
A = \begin{bmatrix}
\lambda & 8  & - 1 \\ 0 & \lambda & 7 \\ 0 & 0 & \lambda
\end{bmatrix}.
\end{align*}
Then \eqref{StD} takes the form  
\begin{align*}
B = \begin{bmatrix}
1 & 8/ \abs{\zeta} & - 1/ \abs{\zeta} \\ 0 & 1 & 7/ \abs{\zeta} \\ 0 & 0 & 1
\end{bmatrix}.
\end{align*}
There is only one positive root, $s \approx 0.7003$, to the corresponding polynomial $P_B$, and indeed for this value of $ \abs{\zeta}$, $St(B) = 0$.  Since $\arg(x\overline{y}z) = \pi$, we then have that $\zeta \approx -0.7003$, implying that $St(A) \approx \lambda - 0.7003$ which is consistent with the value computed directly.

\smallskip
{\bf Example 2.}
Now let 
\begin{align*}
A = \begin{bmatrix}
\lambda & 8 & - 1 \\ 0 & \lambda & 7.5 \\ 0 & 0 & \lambda
\end{bmatrix}.
\end{align*}
Then \eqref{StD} takes the form  
\begin{align*}
B = \begin{bmatrix}
1 & 8/ \abs{\zeta} & - 1/ \abs{\zeta} \\ 0 & 1 & 7/ \abs{\zeta} \\ 0 & 0 & 1
\end{bmatrix}.
\end{align*}
The respective polynomial $P_B$ has three positive roots, approximately equal 0.833, 1.367, and 2.101.
Computations show that $\St(B)$ equals zero only for the matrix $B$ corresponding to the minimal value of $s$. So, $\abs{\zeta} \approx 0.833$.
Since $\arg(x\overline{y}z) = \pi$, it follows that $\zeta \approx -0.833$, implying that $St(A) \approx \lambda - 0.833$. As in Example~1, this is consistent with the value of $\St(A)$ computed directly.

Computation of the displacement $\zeta$ between the spectrum $\{\lambda\}$ of the matrix \eqref{3by3}  and its Stampfli point becomes much simpler under an additional condition $\abs{x}=\abs{z}$. This covers in particular the case $x=z$ in which $A$ is a triangular Toeplitz matrix.

\begin{thm}\label{th:toe}Let in \eqref{3by3} $0\neq \abs{x}=\abs{z}:=u$. Then $\St(A)=\lambda+\zeta$, where $\arg\zeta$ is determined by \eqref{arg1}, while  \eq{abs} \abs{\zeta}= \begin{cases} \frac{u^2v}{u^2-v^2} & \text{ if } \frac{v}{u}\leq\sqrt{7-4\sqrt{3}},\\ 
	\frac{u^2(2\sqrt{6u^2+v^2}-v)}{2(8u^2+v^2)} & \text{ otherwise}, \end{cases} \en
where $\abs{y}:=v$. In particular, \eq{abs1} \abs{\zeta}=\frac{2\sqrt{7}-1}{18}u \text{ if }v=u.\en   \end{thm}
\begin{proof}Observe first of all that \eqref{abs} means that $\zeta=0$ if $y=0$, which is in agreement with Proposition~\ref{th:3by30}. So, only the case $y\neq 0$ is of interest. 
	
Further, under the requirement $\abs{x}=\abs{z}$ in the notation \eqref{uvw} we have $u=w$. Polynomial \eqref{s} then simplifies greatly, and actually factors into the product of $(s v^2 - (s-v) u^2)^2$ by the cubic 
\begin{align*}
4 s^3 v^4 + 4 s^2 (9 s + 2 v) u^2 v^2 + 
s (32 s^2 + 36 s v + v^2) u^4 -3 (s -v) u^6.
\end{align*}
So, it has a double root $u^2v/(u^2-v^2)$ (negative if $u<v$ and disappearing if $u=v$),  and three simple roots 
\[ -\frac{u^2v}{u^2+v^2},\quad  \frac{u^2(-v\pm 2\sqrt{6u^2+v^2})}{2(8u^2+v^2)}, \]
exactly one of which being positive. This justifies the second line of \eqref{abs} when $u\leq v$, and its particular case \eqref{abs1} corresponding to $u=v$. 

It remains to consider the case $u>v$ and show that the matrix 
\[ C= \begin{bmatrix} 1 & \frac{u^2-v^2}{uv} &  -\frac{u^2-v^2}{u^2}\\ 0 & 1 & \frac{u^2-v^2}{uv}\\
	0 & 0 & 1 \end{bmatrix}, \]
obtained from \eqref{B1} by plugging in $\zeta$ from the first line of \eqref{abs} and replacing $w$ with $u$,
satisfies $W_0(C)\ni 0$ if and only if $\frac{v^2}{u^2}\leq 7-4\sqrt{3}$. 

Direct computations show that $C^*C$ has the maximal eigenvalue $u^2/v^2$ of multiplicity two, and the respective eigenspace $\mathcal L$ is the span of 
\[ \xi_1=\begin{bmatrix} -v^2/u^2 \\ 0 \\ 1\end{bmatrix} \text{ and } \xi_2=\begin{bmatrix} v/u \\ 1 \\ 0\end{bmatrix}. \]
Since $\scal{C\xi_2,\xi_2}=2\neq 0$, we need to figure out whether \eq{scC} \scal{C(\xi_1+t\xi_2),\xi_1+t\xi_2}=0\en  for some $t\in\C$. But \[ \scal{C(\xi_1+t\xi_2),\xi_1+t\xi_2}=2\abs{t}^2+1+\frac{v^2}{u^2}-2\re t\frac{v}{u}+\overline{t}\frac{u^2-v^2}{uv}. \] So, in order for \eqref{scC} to hold $t$ must be real and satisfy 
\eq{t} 2t^2+t\left(\frac{u}{v}-3\frac{v}{u}\right)+\left(1+\frac{v^2}{u^2}\right)=0. \en 
Such $t$ exists if and only if the discriminant of the quadratic equation \eqref{t} is non-negative, which amount to 
\[ \frac{v^2}{u^2}-14+\frac{u^2}{v^2}\geq 0,  \]
or, equivalently,  
\[ \frac{v^2}{u^2}\in [0,7-4\sqrt{3}]\cup [7+4\sqrt{3},+\infty). \]
Since $u>v>0$, only the first interval in the right hand side is of relevance. 
\end{proof} 
It is worth mentioning that according to \eqref{abs} $\zeta$ is indeed the smallest positive root of the polynomial \eqref{s} in the case $\abs{x}=\abs{z}$. Here is a numerical example.

\smallskip
{\bf Example 3.} 
Consider the matrix $A$ as follows, 
\begin{align*}
A = \begin{bmatrix}
\lambda & 4  & - 2 \\ 0 & \lambda & 4 \\ 0 & 0 & \lambda
\end{bmatrix}.
\end{align*}
Since $\frac{v}{u} = \frac{1}{2} > \sqrt{7 - 4 \sqrt{3}}$, by \eqref{abs} we have that 
\begin{align*}
    \abs{\zeta} = \frac{4^2(2\sqrt{6 \times 4^2 + 2^2} - 2)}{2(8 \times 4^2 + 2^2)} = \frac{12}{11}
\end{align*}
and $\arg\zeta = \arg(-2) = \pi$ which implies that $\zeta = -\frac{12}{11}$. Hence, $St(A) = \lambda - \frac{12}{11}$.

\section{On 3-by-3 matrices with a two-point spectrum} \label{s:tps}

A 3-by-3 matrix $A$ with a multiple eigenvalue is unitarily similar to 
\eq{me} \begin{bmatrix} \mu & x & y \\ 0 & \mu & z \\ 0 & 0 & \lambda \end{bmatrix}. \en 
This class of matrices is intermediate between \eqref{3by3} and the whole $\C^{3\times 3}$. As such, it does not 
admit (to the best of our knowledge) an explicit procedure for computing $\St(A)$ similar to Theorem~\ref{th:comp}, but still allows for an extension of Proposition~\ref{th:3by30}, delivering the criterion for $\St(A)$ to coincide with the multiple eigenvalue of $A$. We will use the same notation \eqref{uvw} as in Theorem~\ref{th:comp}, in addition abbreviating  $\abs{\lambda-\mu}$ to $\rho$. 
\begin{thm}\label{th:lmm}A matrix $A$ unitary similar to \eqref{me} has $\St(A)$ coinciding with its repeated eigenvalue $\mu$ if and only if the following three conditions hold:  \eq{arglm} \overline{x}y\overline{z}(\lambda-\mu)\leq 0, \en 
\eq{inelm} \abs{\rho v-uw}\leq\sqrt{(u^2-\rho^2)(w^2+\rho^2)}, \en
and \eq{eqlm} vw\rho^3+uv^2\rho^2+vw(v^2+w^2-u^2)\rho-uv^2w^2=0. \en 
\end{thm} 
Note that condition \eqref{inelm} implicitly contains the requirement $u\geq\rho$. 
\begin{proof}  If $\lambda=\mu$, then $\rho=0$ and conditions \eqref{arglm}, \eqref{inelm} hold automatically, while \eqref{eqlm} boils down to $xyz=0$. This is in full agreement with Proposition~\ref{th:3by30}, so we need only to deal with $\lambda\neq\mu$. 

Replacing $A$ of the form \eqref{me} by 
\[ \frac{1}{\lambda-\mu}A-\frac{\mu}{\lambda-\mu}I=\begin{bmatrix} 0 & x/(\lambda-\mu) & y/(\lambda-\mu) 
\\ 0 & 0 & z/(\lambda-\mu) \\ 0 & 0 & 1 \end{bmatrix},\] 
we may without loss of generality suppose that in \eqref{me} $\lambda=1$ and $\mu=0$. In other words, it suffices to prove the statement for matrices
\eq{me1} A= \begin{bmatrix} 0 & x & y \\ 0 & 0 & z \\ 0 & 0 & 1 \end{bmatrix}, \en 
where in addition we may (as was done earlier) assume $x,z\geq 0$. Respectively, 
conditions \eqref{arglm}--\eqref{eqlm} for the purpose of this proof should be replaced by 
\eq{inelm1} x\geq 1, \ y\leq 0, \ \abs{xz+y}\leq\sqrt{(x^2-1)(z^2+1)} \en 
 and \eq{eqlm1} yz-xy^2+yz(y^2+z^2-x^2)+xy^2z^2=0. \en 

According to \eqref{me1}, \eq{a*a1} A^*A=[0]\oplus \begin{bmatrix} x^2 & xy \\ x\overline{y} & \abs{y}^2+z^2+1 \end{bmatrix}. \en
Eigenvectors $\xi$ corresponding to the maximal eigenvalue $\sigma$ of $A^*A$ therefore have the first coordinate equal to zero. Consequently, \eq{sca} A\xi = [x\xi_2+y\xi_3, z\xi_3, \xi_3]^T, \text{ and } \scal{A\xi,\xi}=\xi_3(z\overline{\xi_2}+\overline{\xi_3}). \en

{\sl Case 1.} $x=0$. From \eqref{a*a1} we see that $\xi$ is collinear with $e_3$, and 
\eqref{sca} implies $\scal{A\xi,\xi}=\abs{\xi_3}^2\neq 0$. So, $\St(A)\neq 0$. Since condition \eqref{inelm1} fails, the statement holds.  

{\sl Case 2.} $y=0$. Then condition \eqref{eqlm1} holds while \eqref{inelm1} simplifies to $x^2\geq 1+z^2$. On the other hand, in this case $A^*A=\diag[0,x^2,z^2+1]$. So, $\xi$ is collinear to $e_3$ if $x^2 < 1+z^2$ while $\xi=e_2$ is admissible otherwise. By \eqref{sca}, $W_0(A)\ni 0$ is equivalent to $x^2\geq 1+z^2$. Once again, the statement holds. 

It remains to consider 

{\sl Case 3.} $x,y\neq 0$. The eigenvalue $\sigma$ of $A^*A$ is then simple, and the entries $\xi_2,\xi_3$ of the respective eigenvector $\xi$ are different from zero. By scaling, without loss of generality let $\xi_2=1$. According to \eqref{sca}, condition $\scal{A\xi,\xi}=0$ then holds if and only if $\xi_3=-z$. Rewritten entry-wise, $A^*A\xi=\sigma\xi$ takes the form
\eq{ev} x^2-xyz=\sigma, \quad x\overline{y}-(\abs{y}^2+z^2+1)z=-\sigma z. \en 

Since $\sigma>\abs{y}^2+z^2+1$, the second equality in \eqref{ev} implies that $y<0$, and \eqref{eqlm1} follows. Finally, $\tau=\tr(A^*A)-\sigma= 1+y^2+z^2+xyz$ is the second non-zero eigenvalue of $A^*A$, and the upper bound on $\abs{xz+y}$ in \eqref{inelm1} is necessary and sufficient for $\tau$ not to exceed $\sigma$.
\end{proof} 

\section{Roberts orthogonality} \label{s:rorth}

Two vectors $x$ and $y$ of a normed linear space $X$ are called {\em Roberts orthogonal} (denoted: $x\perp_R y$) if for all scalars $\nu$: \[ \norm{x+\nu y}=\norm{x-\nu y}. \]
If $X$ is an inner product space, Roberts orthogonality and usual orthogonality coincide: $x\perp_R y$ if and only if $\scal{x,y}=0$. While this notion goes back to \cite{Rob34}, a more recent treatment of the case when $X$ is a unital $C^*$-algebra can be found in \cite{ArBeRa}. In particular, \cite[Theorem 2.4]{ArBeRa} delivers the criterion for $A\in B(\mathcal H)$ to be Roberts orthogonal to the identity operator $I$. To describe this result, recall that the {\em Davis-Wielandt shell} of $A$ is the subset of $\C\times\R$ (identified with $\R^3)$ defined by \[ DW(A)=\{ (\scal{Ax,x}, \norm{Ax}^2)\colon x\in\mathcal H, \norm{x}=1\}.\]
This set is a (possibly, degenerate) ellipsoid if $\dim\mathcal H=2$, and a convex body otherwise, closed if $\dim\mathcal H<\infty$, and located between the horizontal planes $\Pi_0=\C\times\{0\}$ and $\Pi_1=\C\times\{\norm{A}^2\}$. 

For each vertical line having a non-empty intersection with $\cl({DW(A)})$ choose the uppermost point of this intersection, and call the union $DW_{ub}(A)$ of all such points the {\em upper boundary} of $DW(A)$. 
It was shown in \cite{ArBeRa} that $A\perp_R I$ if and only if $DW_{ub}(A)=DW_{ub}(-A)$; in other words, if and only if $DW_{ub}(A)$ is symmetric with respect to the vertical coordinate axis. 

Note however that $\cl{W(A)}$ and $W_0(A)$ are the projections onto $\Pi_0$ of $DW_{ub}(A)$ and its intersection with $\Pi_1$, respectively. So, from $DW_{ub}(A)=DW_{ub}(-A)$ it immediately follows that both $\cl{W(A)}$ and $W_0(A)$ are central symmetric. The necessity of the former condition for $A$ and $I$  to be Roberts orthogonal (along with its sufficiency when $\dim\mathcal H=2$) was observed in \cite{ArBeRa}. From the necessity of the latter and \eqref{stw0} we obtain 
\begin{prop}\label{pr:rost}Let $A\in B(\mathcal H)$ be Roberts orthogonal to $I$. Then $\St(A)=0$.  \end{prop} 
Indeed, the set $W_0(A)$ is convex, and its central symmetry implies therefore that it contains zero. 

\begin{thm}\label{th:rq}A quadratic operator $A$ is Roberts orthogonal to the identity if and only if it is nilpotent or a scalar multiple of an involution. \end{thm}
\begin{proof}{\sl Necessity.} Combining Proposition~\ref{pr:rost} with Theorem~\ref{th:quad} we see that $A$ satisfies \eqref{quad} with $p=0$. If in addition $q=0$, then $A$ is nilpotent; otherwise $(-q)^{-1/2}A$ is an involution.

{\sl Sufficiency.} The operator in question is unitarily similar to \eqref{mquad} with $\lambda_1=-\lambda_2:=a$. The observation $\norm{A-\lambda I}=\norm{A_0-\lambda I}$ from the proof of Theorem~\ref{th:quad} followed by the application of \eqref{2max} to $A_0$ in place of $A$ show that in our case 
\[ \norm{A-\lambda I}^2=\abs{a}^2+\abs{\lambda}^2+\norm{Z}^2/2
+\sqrt{(\abs{a}^2+\abs{\lambda}^2+\norm{Z}^2/2)^2-\abs{a^2-\lambda^2}^2}. \]  
The right hand side is indeed invariant under the change $\lambda\mapsto -\lambda$, and so $A\perp_R I$. \end{proof} 
So, for quadratic operators $A$ condition $\St(A)=0$ is not only necessary but also sufficient for $A\perp_R I$. As a by-product we see that the central symmetry of $W(A)$ is yet another equivalent condition; the fact for $n=2$ observed in \cite[Proposition~2.7]{ArBeRa}

Moving to $A\in\C^{3\times 3}$, we restrict our attention to matrices with circular numerical ranges.

\begin{thm}\label{th:r3} A matrix $A\in\C^{3\times 3}$ with a circular numerical range is Roberts orthogonal to the identity if and only if it is nilpotent or unitarily similar to a direct sum of a $1$-by-$1$ and a nilpotent $2$-by-$2$ block.  \end{thm} 
\begin{proof}
{\sl Necessity.} Let  $A\perp_R I$ while $W(A)$ is a circular disk. The central symmetry of $W(A)$ means that this disk is centered at the origin. According to \cite[Corollary 2.5]{KRS} (see also \cite{CT94}), such $A$ is unitarily similar to \eq{Acd} \begin{bmatrix} 0 & x & y \\ 0 & 0 & z \\ 0 & 0 &  \lambda \end{bmatrix}, \en 
where \eq{c1} x\overline{y}z=-\lambda(\abs{y}^2+\abs{z}^2)\en   and
\eq{c2} \abs{x}^2+\abs{y}^2+\abs{z}^2\geq 4\abs{\lambda}^2. \en 
In the notation of Theorem~\ref{th:lmm} condition \eqref{c1} implies $uvw=\rho(v^2+w^2)$. Plugging this into \eqref{eqlm} (which holds, since $\St(A)=\mu=0$) we conclude that $\lambda=0$ or $z=0$. If $\lambda=0$, the matrix \eqref{Acd} is nilpotent. Otherwise, \eqref{c1} implies that along with $z$ also $y=0$, and $A$ is therefore unitarily similar to \eq{Ar} [\lambda]\oplus B, \text{ where } B=\begin{bmatrix} 0 & x \\ 0 & 0\end{bmatrix}. \en  

{\sl Sufficiency.} If $A$ is nilpotent, the circularity of its numerical range implies that in \eqref{Acd} not only $\lambda=0$, but also $xyz=0$, due to \eqref{c1}. For any $\nu\in\C$ the matrix $A-\nu I$ by a rotation through $-\arg\nu$ and an appropriate diagonal unitary similarity can be reduced to 
\[ \begin{bmatrix} \abs{\nu} & \abs{x} & \abs{y} \\ 0 & \abs{\nu} & \abs{z}\\ 0 & 0 & \abs{\nu} \end{bmatrix}. \] 
So, $\norm{A-\nu I}$ in this case does not depend on the argument of $\mu$, only on its absolute value. In particular, $A\perp_R I$. 

Finally, if $A$ is unitarily similar to \eqref{Ar}, the circularity of $W(A)$ implies that $2\abs{\lambda}\leq\abs{x}$.
From here, 
\[ \norm{B-\nu I}^2=\abs{\nu}^2+\frac{{\abs{x}}^2}{2}+\sqrt{\abs{x\nu}^2+\frac{\abs{x}^4}{4}}\geq 
\abs{\nu}^2+2{\abs{\lambda}^2+2\abs{\lambda}\sqrt{\abs{\nu}^2+\frac{\abs{\lambda}^2}{4}}}, \]
so $\abs{\lambda-\nu}\leq\norm{B-\nu I}$ for all $\nu\in\C$. 
Consequently, \[ \norm{A-\nu I}=\max\{\abs{\lambda-\nu}, \norm{B-\nu I}\}= \norm{B-\nu I},\] 
and the relation $A\perp_R I$ follows from  $B\perp_R I$.
\end{proof} 
\begin{cor}\label{co:cst}Let $A\in\C^{3\times 3}$ have a circular numerical range. Then $A\perp_RI$ if and only if the disk $W(A)$ is centered at zero and $\St(A)=0$.  \end{cor}
A particular example \[ A= \begin{bmatrix} 0 & 1 & 1 \\ 0 & 0 & 1\\ 0 & 0 & -\frac{1}{2} \end{bmatrix} \]
was considered in \cite{ArBeRa}. It was shown there that $A\not\perp_R I$ via computing \[ \norm{A+I}\approx 2.1617\neq 2.1366\approx \norm{A-I}, \] while $W(A)$ is a circular disk centered at the origin. This agrees with our Theorem~\ref{th:r3} since this matrix is neither unitarily reducible nor nilpotent. Note also that $\St(A)\approx 0.0203\neq 0$, in agreement with Corollary~\ref{co:cst}.

\section*{Appendix} \label{s:ap}
The figures below represent graphs of numerical ranges (bounded by green curves), spectra (blue dots) and Stampfli points (black dots) of several matrices, to illustrate some statement of the paper. The matrices in Fig.~1, 2 and 4 are nilpotent and unitarily irreducible, with the numerical range having each of the three possible shapes (circular, with a flat portion on the boundary, or ovular, as per \cite{KRS}). The one with the circular numerical range (Fig.~1) satisfies conditions of Proposition~\ref{th:3by30}, and indeed has the Stampfli point located at the origin. In Fig.~3 and~4, the Stampfli point differs from zero, and is positioned in agreement with Theorem~\ref{th:comp}. Finally, Fig.~2 illustrates that for almost normal matrices, $\St(A)$ lies in the interior of $\conv\sigma(A)$ (bounded by the red triangle), in agreement with Corollary~\ref{co:stanf}.

\begin{figure}[htbp]
\centerline{\includegraphics[scale= 0.42]{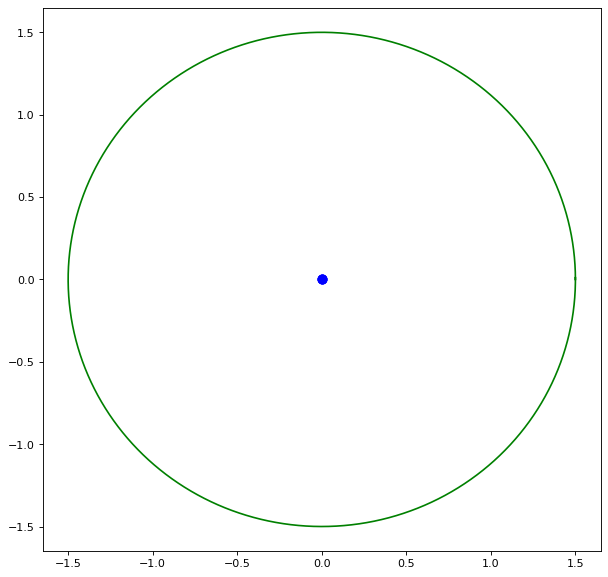}}
\caption{$A =\protect\begin{bmatrix} 0 & 2 - i & 0 \protect\\ 0 & 0 & 2i \protect\\ 0 & 0 & 0 \protect\end{bmatrix}$; $St(A) = 0$.}
\end{figure}

\begin{figure}[htbp]
\centerline{\includegraphics[scale= 0.42]{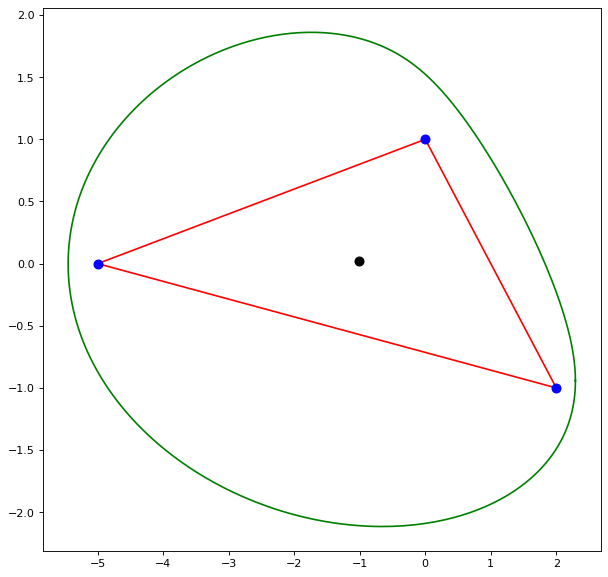}}
\caption{$A =\protect\begin{bmatrix} 2 + i & 0 & 2 - 2i \protect\\ 0 & i & 2 \protect\\ 0 & 0 & -5 \protect\end{bmatrix}$; $St(A) = -1.008 + 0.0237i$}
\end{figure}

\begin{figure}[htbp]
\centerline{\includegraphics[scale= 0.42]{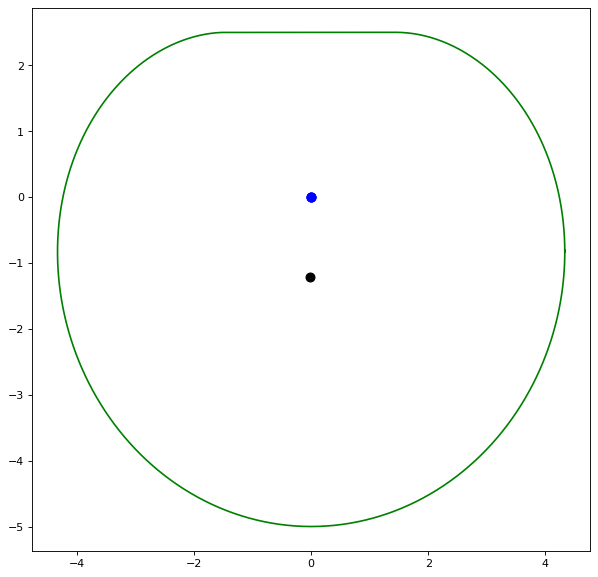}}
\caption{$A =\protect\begin{bmatrix} 0 & 3 - 4i & -5 \protect\\ 0 & 0 & -4+3i \protect\\ 0 & 0 & 0 \protect\end{bmatrix}$; $St(A) = -0.0145 -1.2143i$}
\end{figure}

\begin{figure}[htbp]
\centerline{\includegraphics[scale= 0.42]{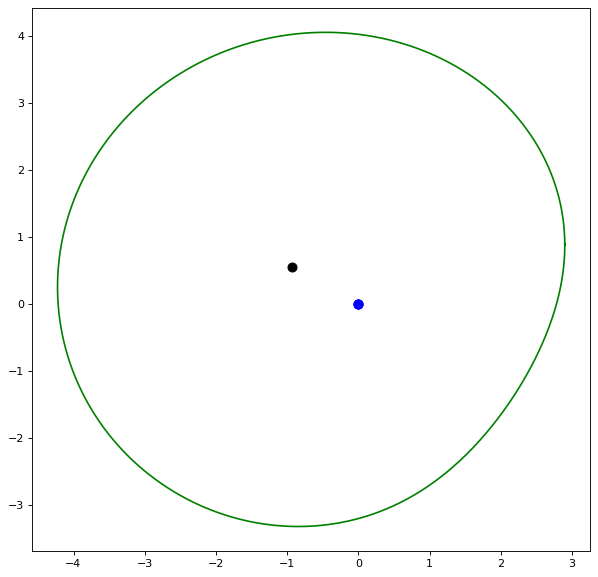}}
\caption{$A =\protect\begin{bmatrix} 0 & 1 - 4i & -3 - 2i \protect\\ 0 & 0 & 1 + 5i \protect\\ 0 & 0 & 0 \protect\end{bmatrix}$; $St(A) = -0.9363 + 0.5225i$}
\end{figure}

\clearpage
\providecommand{\bysame}{\leavevmode\hbox to3em{\hrulefill}\thinspace}
\providecommand{\MR}{\relax\ifhmode\unskip\space\fi MR }
\providecommand{\MRhref}[2]{%
	\href{http://www.ams.org/mathscinet-getitem?mr=#1}{#2}
}
\providecommand{\href}[2]{#2}

\end{document}